\documentclass[12pt]{article}
\usepackage{amsmath,amssymb,amscd,latexsym,bm,paralist}
\usepackage[all]{xy}

\newcommand{\C}{{\mathbb{C}}}
\newcommand{\F}{{\mathbb{F}}}
\newcommand{\G}{\mathbb{G}}

\newcommand{\Q}{{\mathbb{Q}}}
\newcommand{\oQ}{\overline{\Q}}

\newcommand{\R}{{\mathbb{R}}}
\newcommand{\Z}{{\mathbb{Z}}}

\newcommand{\hM}{\hat{M}}

\newcommand{\hX}{\hat{X}}
\newcommand{\oGamma}{\overline{\Gamma}}

\newcommand{\hGamma}{\hat{\Gamma}}
\newcommand{\hPhi}{\hat{\Phi}}

\newcommand{\abb}{\mathrm{ab}}

\newcommand{\ddet}{\mathrm{det}}

\newcommand{\of}{\overline{f}}
\newcommand{\og}{\overline{g}}

\renewcommand{\mod}{\;\mathrm{mod}\;}

\newcommand{\per}{\mathrm{per}}

\newcommand{\spec}{\mathrm{spec}\,}
\newcommand{\syn}{\mathrm{syn}}
\newcommand{\End}{\mathrm{End}\,}
\newcommand{\Fix}{\mathrm{Fix}}

\newcommand{\GL}{\mathrm{GL}\,}

\newcommand{\Ker}{\mathrm{Ker}\,}

\newcommand{\rig}{\mathrm{rig}}

\newcommand{\tr}{\mathrm{tr}}
\newcommand{\Ah}{{\mathcal A}}

\newcommand{\Dh}{{\mathcal D}}

\newcommand{\Kh}{\mathcal{K}}

\newcommand{\Mh}{\mathcal{M}}
\newcommand{\Nh}{\mathcal{N}}
\newcommand{\Oh}{{\mathcal O}}

\newcommand{\ep}{\mathfrak{p}}

\newcommand{\ozeta}{\overline{\zeta}}

\newcommand{\hf}{\hat{f}}

\newcommand{\silo}{\xrightarrow{\sim}}

\newcommand{\verk}{\mbox{\scriptsize $\,\circ\,$}}

\newtheorem{theorem}{Theorem}

\newtheorem{prop}[theorem]{Proposition}
\newtheorem{cor}[theorem]{Corollary}

\newtheorem{remark}[theorem]{Remark}

\newenvironment{proof}{\noindent {\bf Proof}}{\mbox{}\hfill$\Box$}
\begin{document}
\title{\Large Regulators, entropy and infinite determinants}
\author{\large Christopher Deninger}
\date{\ }
\maketitle
\section{Introduction}
In this note we describe instances where values of the $K$-theoretical regulator map evaluated on topological cycles equal entropies of topological actions by a group $\Gamma$. These entropies can also be described by determinants on the von~Neumann algebra of $\Gamma$. There is no conceptual understanding why regulators should be related to entropies or to such determinants and it would be interesting to find an explanation. The relations were first observed for real regulators. The latter have $p$-adic analogues and both $p$-adic entropy and $p$-adic determinants were then defined so that similar relations hold as in the real case. We describe this $p$-adic theory in the second part of the paper. In fact there are two natural $p$-adic analogues for the relevant values of the real regulator. A purely local one which we use and a local-global one for which as yet no corresponding notions of $p$-adic entropy or $p$-adic infinite determinants exist.

In both the real and $p$-adic cases our regulator maps are defined on subvarieties of $\spec \Q [\Gamma]$ where $\Gamma = \Z^d$. Entropy and determinants on the other hand are defined for nonabelian groups $\Gamma$ as well. We explain this in some detail because it would be interesting to find a corresponding nonabelian generalization of the classical and $p$-adic regulator in some version of non-commutative algebraic geometry. For the Heisenberg group for example, the expected regulator values are quite explicitely known by dynamical considerations.

This note is almost entirely a survey of known results with the exception of some results in section \ref{subsec:31}. However the different aspects of the theory have not been discussed together before. Along the way we point out several open questions and possible directions for further research. 

\section{The archimedian case} \label{sec:2}
\subsection{Expansiveness and entropy} \label{subsec:21}
Consider a countable discrete group $\Gamma$. A left action of $\Gamma$ by homeomorphisms on a compact metrizable topological space $X$ is called expansive if the following holds: There is a metric $d$ defining the topology of $X$ and some $\varepsilon > 0$ such that for every pair of distinct points $x \neq y$ in $X$ there exists an element $\gamma \in \Gamma$ with
\begin{equation} \label{eq:1}
 d (\gamma x , \gamma y) \ge \varepsilon \; .
\end{equation}
If condition \eqref{eq:1} holds for one metric $d$ it holds for every other metric $d'$ as well, possibly with a different constant $\varepsilon$. We will meet many examples of expansive actions in section \ref{subsec:23}. 

Next we need the notion of entropy. Assume that in addition $\Gamma$ is amenable. This is equivalent to the existence of a F{\o}lner sequence $F_1 , F_2 , \ldots $ of finite subsets of $\Gamma$ such that for every $\gamma \in \Gamma$ we have
\[
 \lim_{n\to \infty} \frac{|F_n \gamma \bigtriangleup F_n|}{|F_n|} = 0 \; .
\]
For finite $F \subset \Gamma$, a subset $E$ of $X$ is called $(F, \varepsilon)$-separated if for all $x \neq y$ in $E$ there exists $\gamma \in \Gamma$ with $d (\gamma  x , \gamma y) \ge \varepsilon$. Let $s_F (\varepsilon)$ be the maximum of the cardinalities of all $(F, \varepsilon)$-separated subsets. It is finite because $X$ is compact. The (metric) entropy $0 \le h \le \infty$ of the $\Gamma$-action on $X$ is defined by the formula:
\[
 h = h (X) := \lim_{\varepsilon \to 0} \varlimsup_{n\to \infty} \frac{1}{|F_n|} \log s_{F_n} (\varepsilon) \; .
\]
It depends neither on the choice of a metric defining the topology of $X$ nor on the choice of the F{\o}lner sequence. Note that the $\varepsilon$-limit exists by monotonicity. It becomes stationary if the $\Gamma$-action is expansive. See \cite{} section 2 for some of the proofs and further references. The sequence $F_n = \{ 0 , 1 , 2 , \ldots , n \}^d$ is a F{\o}lner sequence in $\Gamma = \Z^d$ and hence the entropy of $\Z^d$-actions is defined. Note that if both $\Gamma$ and $X$ are finite, $F_i = \Gamma$ for $i = 1 , 2 , \ldots$ is a F{\o}lner sequence and we have
\begin{equation} \label{eq:2}
 h = \frac{1}{|\Gamma|} \log |X| \; .
\end{equation}
Although trivial from the point of view of dynamical systems it is useful to keep this example in mind. For expansive actions the entropy can sometimes be determined by counting periodic points. Recall that a group $\Gamma$ is residually finite if it has a sequence $\Gamma_1 , \Gamma_2 , \ldots $ of normal subgroups of finite index with trivial intersection. Let us write $\Gamma_n \to e$ in case the sequence $(\Gamma_n)$ satisfies the stronger condition that only the neutral elements $e$ lies in $\Gamma_n$ for infinitely many $n$'s. Such sequences exist as well. Let $\Fix_{\Gamma_n} (X)$ be the set of points in $X$ which are fixed by $\Gamma_n$. If the action of $\Gamma$ on $X$ is expansive, $\Fix_{\Gamma_n} (X)$ is finite. Set
\begin{equation} \label{eq:3}
 h^{\per} = \lim_{n\to\infty} \frac{1}{(\Gamma : \Gamma_n)} \log |\Fix_{\Gamma_n} (X)|
\end{equation}
if the limit exists for any choice of sequence $\Gamma_n \to e$. It is then independent of the choice of $\Gamma_n$'s. Note that $\Gamma = \Z^d$ is residually finite and that $\Gamma_n = (n \Z)^d$ provides a sequence with $\Gamma_n \to e = 0$. As mentioned above, it sometimes happens for expansive actions of countable discrete groups which are both residually finite and amenable that we have $h = h^{\per}$. As a trivial example note that because of \eqref{eq:2} this is true if both $\Gamma$ and $X$ are finite. In this case the condition $\Gamma_n \to e$ means that $\Gamma_n = \{ e \}$ if $n$ is large enough.  More interesting examples are given in section \ref{subsec:23}.

\subsection{Determinants on von~Neumann algebras} \label{subsec:22}
Let $\Mh$ be a von~Neumann algebra with a faithful finite normal trace $\tau$, see \cite{Di}. For an operator $\Phi$ in $\Mh$ let $E_{\lambda} = E_{\lambda} (|\Phi|)$ be the spectral resolution of the selfadjoint operator $|\Phi| = (\Phi^* \Phi)^{1/2}$. Then both $|\Phi|$ and $E_{\lambda}$ lie in $\Mh$ and one defines the Fuglede--Kadison determinant $\det_{\Mh} \Phi \ge 0$ by the integral
\[
 \log \ddet_{\Mh} \Phi = \int^{\infty}_0 \log \lambda \, d\tau (E_{\lambda}) \quad \text{in} \; \R \cup \{ -\infty \} \; . 
\]
For an invertible operator $\Phi \in \Mh^{\times}$ we have
\[
 \log \ddet_{\Mh} \Phi = \tau (\log |\Phi|) \; .
\]
The general case can be reduced to this one by the formula
\[
 \ddet_{\Mh} \Phi = \lim_{\varepsilon \to 0+} \ddet_{\Mh} (|\Phi| + \varepsilon) \; .
\]
It is a non-obvious fact that this determinant is multiplicative
\[
 \ddet_{\Mh} \Phi_1 \Phi_2 = \ddet_{\Mh} \Phi_1 \; \ddet_{\Mh} \Phi_2 \quad \text{for} \; \Phi_1 , \Phi_2 \; \text{in} \; \Mh \; .
\]
Morally this is due to the Campbell--Hausdorff formula although the proof in \cite{Di} I.6.11 proceeds differently using a uniqueness result for solutions of ordinary differential equations.

The determinant on $\Mh^{\times}$ factors over $K_1 (\Mh) = \GL_{\infty} (\Mh)^{\abb}$. Namely, for $n \ge 1$ consider the von~Neumann algebra $M_n (\Mh)$ with  trace $\tau_n = \tau \verk \tr_n$ where $\tr_n : M_n (\Mh) \to \Mh$ is the usual trace of matrices. The determinants on $\GL_n (\Mh) = M_n (\Mh)^{\times}$ are compatible for varying $n$ and define a homomorphism on $\GL_{\infty} (\Mh)$ which factors over $K_1 (\Mh)$
\begin{equation} \label{eq:3a}
 \ddet_{\Mh} : K_1 (\Mh) \longrightarrow \R^*_+\; .
\end{equation}

For the abelian von~Neumann algebra $\Mh = L^{\infty} (\Omega)$ attached to a finite measure space $(\Omega , \Ah , \mu)$ and equipped with the trace $\tau (\Phi) = \int_{\Omega} \Phi  \, d\mu $ we have
\begin{equation} \label{eq:4}
 \ddet_{\Mh} \Phi = \exp \int_{\Omega} \log |\Phi| \, d \mu  \quad \text{for any} \; \Phi \in \Mh \; .
\end{equation}
This follows immediately from the definitions. Our von~Neumann algebras will arise from groups. For a discrete group let $L^2 (\Gamma)$ be the Hilbert space of square summable complex valued functions $x : \Gamma \to \C$. The group $\Gamma$ acts isometrically from the right on $L^2 (\Gamma)$ by the formula $(x\gamma) (\gamma') = x (\gamma' \gamma^{-1})$. Representing elements of $L^2 (\Gamma)$ as formal sums $\sum x_{\gamma'} \gamma'$, this corresponds to right multiplication by $\gamma$. The von~Neumann algebra of $\Gamma$ is the algebra of $\Gamma$-equivariant bounded linear operators from $L^2 (\Gamma)$ to itself. It is equipped with the faithful finite normal trace $\tau_{\Gamma} : \Nh\Gamma \to \C$ defined by $\tau_{\Gamma} (\Phi) = (\Phi (e) , e)$ where $(,)$ is the scalar product on $L^2 (\Gamma)$ and $e \in \Gamma \subset L^2 (\Gamma)$ is the unit of $\Gamma$. For $\gamma \in \Gamma$ consider the unitary operator $\ell_{\gamma}$ of left multiplication on $L^2 (\Gamma)$. It commutes with the right $\Gamma$-multiplication and hence defines an element of $\Nh\Gamma$. The $\C$-algebra homomorphism
\[
 \ell : \C \Gamma \longrightarrow \Nh \Gamma \; , \; {\textstyle\sum} a_{\gamma} \gamma \longmapsto {\textstyle\sum} a_{\gamma} \ell_{\gamma}
\]
is injective as one sees by evaluating $\ell (f)$ on $e$. It will always be viewed as an inclusion. It is not difficult to see that $\ell$ extends to the $L^1$-group algebra $L^1 (\Gamma)$ of $\Gamma$, so that we have inclusions of $\C$-algebras
\[
 \C \Gamma \subset L^1 (\Gamma) \subset \Nh \Gamma \; .
\]
In particular, an element $f$ of $\C \Gamma$ which is a unit in $L^1 (\Gamma)$ is also a unit in $\Nh\Gamma$. We remark this because in the $p$-adic case we do not have a good replacement for $\Nh\Gamma$ with its many idempotents but only one for $L^1 (\Gamma)$. Writing an element $f \in L^1 (\Gamma)$ as a formal series $f = \sum a_{\gamma} \gamma$ with $\sum |a_{\gamma}| < \infty$ note that $\tau_{\Gamma} (f) = a_e$. In the trivial case of a finite group we have $\Nh\Gamma \subset \End \C \Gamma$ and
\[
 \ddet_{\Nh\Gamma} f = |\ddet \, \ell (f)|^{1 / |\Gamma|} \; .
\]
For a discrete abelian group $\Gamma$ the von~Neumann algebra is easy to describe using the compact Pontrjagin dual group $\hGamma$ with its Haar probability measure $\mu$. Namely, the Fourier transform gives an isometry $\wedge : L^2 (\Gamma) \silo L^2 (\hGamma)$. Conjugation by $\wedge$ gives an isomorphism of $\Nh\Gamma$ with the algebra $L^{\infty} (\hGamma)$ of multiplication operators on $L^2 (\hGamma)$. Denoting this isomorphism $\Nh\Gamma \to L^{\infty} (\hGamma)$ by $\Phi \mapsto \hPhi$, we have $\hPhi = \widehat{\Phi (e)}$. On the other hand, the Fourier transform restricts to a map 
\[
\wedge : L^1 (\Gamma) \to L^{\infty} (\hGamma) \subset L^2 (\hGamma) \; .
\]
Setting $\tau (\varphi) = \int_{\hGamma} \varphi \, d\mu$, we get a commutative diagram 
\[
 \xymatrix{
L^1 (\Gamma) \ar@{^{(}->}[r]^{\ell} \ar@{^{(}->}[d]^{\wedge} & \Nh \Gamma \ar[r]^{\tau_{\Gamma}} \ar[d]^{\wr \wedge} & \C \ar@{=}[d] \\
L^{\infty} (\hGamma) \ar@{=}[r] & L^{\infty} (\hGamma) \ar[r]^{\tau} & \C \; .
}
\]
Using \eqref{eq:4}, this implies the formula:
\begin{equation} \label{eq:5}
 \ddet_{\Nh\Gamma} f = \exp \int_{\hGamma} \log |\hf| \, d\mu \quad \text{for} \; f \in L^1 (\Gamma) \; .
\end{equation}
Sometimes the Fuglede--Kadison determinant for a group $\Gamma$ may be calculated as a renormalized limit of finite determinants. Here is a simple case where this is possible, \cite{DS} Theorem 5.7. For a residually finite group $\Gamma$ and a sequence $\Gamma_n \to e$ as above write $\Gamma^{(n)} = \Gamma / \Gamma_n$ and let
\[
 L^1 (\Gamma) \longrightarrow L^1 (\Gamma^{(n)}) = \C \Gamma^{(n)} \; , \; f \longmapsto f^{(n)}
\]
be the algebra homomorphism of ``integration along the fibres'' : For $f = \sum a_{\gamma} \gamma$ in $L^1 (\Gamma)$ it is defined by the formula:
\[
 f^{(n)} = \sum_{\delta \in \Gamma^{(n)}} \Big( \sum_{\gamma \in \delta} a_{\gamma} \Big) \delta \; .
\]

\begin{theorem} \label{t1}
 If $f$ is a unit in $L^1 (\Gamma)$, we have:
\[
 \ddet_{\Gamma} f = \lim_{n\to\infty} \ddet_{\Gamma^{(n)}} f^{(n)}  = \lim_{n\to\infty} |\ddet \, \ell (f^{(n)})|^{1 / |\Gamma ^{(n)}|} \; .
\]
\end{theorem}

Consider the example $\Gamma = \Z^d$. In this case $\hGamma = T^d$ is the real $d$-torus in $(\C^*)^d$. Writing the coordinates of $\C^d$ as $z_1 , \ldots , z_d$, Fourier transform gives an isomorphism $\wedge : \C [\Z^d] \silo \C [z^{\pm 1}_1 , \ldots , z^{\pm 1}_d] \subset L^{\infty} (T^d)$ of the group ring with the ring of Laurent polynomials in $z_1 , \ldots , z_d$. The logarithm of the FK-determinant of $f \in \C [\Z^d]$ equals the (logarithmic) Mahler measure of $\hf$
\begin{equation} \label{eq:6}
 \log \ddet_{\Nh\Z^d} (f) = m (\hf) := \int_{T^d} \log |\hf| \, d\mu \; .
\end{equation}
We have the following theorem of Wiener, \cite{W} Lemma IIe. A modern proof is given in \cite{K}.

\begin{theorem} \label{t2}
	The element $f$ of $L^1 (\Z^d)$ is a unit in $L^1 (\Z^d)$ if and only if $\hf (z) \neq 0$ for all $z \in T^d$.
\end{theorem}

Equivalently, since $\hf$ is continuous for $f \in L^1 (\Z^d)$ and since $\Nh\Z^d = L^{\infty} (T^d)$, the theorem asserts that $f$ is a unit in $L^1 (\Z^d)$ if and only if it is a unit in $\Nh \Z^d$. Accordingly one says that a group $\Gamma$ has the Wiener property if for any element $f \in L^1 (\Gamma)$, being a unit in $L^1 (\Gamma)$ is equivalent to being a unit in $\Nh\Gamma$. For example, finitely generated nilpotent groups have this property \cite{Lo} but in general it does not hold. 

In any case, for $\Gamma = \Z^d$ and $\Gamma_n = (n\Z)^d$, theorem \ref{t1} asserts that if the Fourier series $\hf$ of $f \in L^1 (\Z^d)$ does not vanish in any point of $T^d$ we have the formula:
\begin{equation} \label{eq:7}
 \log \ddet_{\Nh\Z^d} f = \lim_{n\to\infty} n^{-d} \sum_{\zeta \in \mu^d_n} \log |\hf (\zeta)| \; .
\end{equation}
Here we have used formula \eqref{eq:5} for $\Gamma = (\Z / n)^d$ noting that the Pontrjagin dual of $\Z / n$ is identified with the group of $n$-th roots of unity $\mu_n$. Of course, using Riemann sums one sees that as a formula for the Mahler measure, equation \eqref{eq:7} holds more generally for all continuous functions $\hf$ on $T^d$ without zeroes. 

It is possible that the assertion of theorem \ref{t2} is valid for all $f \in \Z \Gamma$ without the condition that $f \in L^1 (\Gamma)^{\times}$. For $\Gamma =\Z^d$ this is true by non-trivial diophantine results, \cite{Lin}. For $\Gamma = \Z$, by Jensen's formula the Mahler measure of $f \in \C [\Z]$ can be described in terms of the zeroes of the polynomial $\hf \in \C [z, z^{-1}]$. Writing $\hf (z) = a_m z^m + \ldots + a_r z^r$ with $a_m , a_r$ non-zero, we have:
\begin{align}
 \log \ddet_{\Nh\Z} f = m (\hf) & = \log |a_r| - \sum_{0 < |\alpha| < 1} \log |\alpha| \label{eq:8}\\
& = \log |a_m| + \sum_{|\alpha| > 1} \log |\alpha| \; . \nonumber
\end{align}
Here $\alpha$ runs over the zeroes of $\hf$ with their multiplicities. Despite this seemingly easy decription there are difficult unsolved issues about the Mahler measures of polynomials with integer coefficients. We refer to the survey article \cite{Sm} for background. For Mahler measures of polynomials in several variables there are no simple formulas. However in several cases Mahler measures for suitable $\hf$ in $\Z [z^{\pm 1}_1 , \ldots , z^{\pm} _d]$ are related to special values of $L$-series. See \cite{Bo}, \cite{La}, \cite{RV}. The relation to regulators which we will recall in section \ref{subsec:24} was the starting point for explaining this phenomenon in terms of the Beilinson conjectures.

We end this section with an example where the FK-determinant of a non-commutative group can be calculated in terms of ordinary integrals \cite{D4}. The corresponding formula for entropies was discovered by Lind and Schmidt. Let $\Gamma$ be the integral Heisenberg group with generators $x,y$ and central commutator $z = y^{-1} x^{-1} yx$. For complex polynomials $a_i (y,z)$ in the commuting variables $y$ and $z$ we can form the non-commutative polynomial
\[
 f = \sum^N_{i=0} a_i (y,z) x^i \quad \text{in} \; \C\Gamma \; .
\]
In general $\log \ddet_{\Nh\Gamma} f$ can be expressed as an integral over $S^1$ over certain Ljapunov exponents associated to the companion matrix of $f$, see \cite{D4}. In the simple case where $f = 1 - a (y,z) x$ this reduces to the explicit formula:
\begin{equation} \label{eq:9}
 \log \ddet_{\Nh\Gamma} f = \int_{S^1} \Big( \int_{S^1} \log |a (\eta , \zeta)| \, d\mu (\eta) \Big)^+ \, d\mu (\zeta) \; .
\end{equation}
Here $t^+ = \max (t,0)$. Incidentally, it follows from Jensen's formula that if we view $f$ as an element of $\C [\Z^3]$ we obtain instead the formula:
\begin{equation} \label{eq:10}
 \log \ddet_{\Nh \Z^3} f = \int_{S^1} \int_{S^1} \log^+ |a (\eta , \zeta)| \, d\mu (\eta) \, d\mu (\zeta) \; .
\end{equation}
It follows that for such $f$ we have
\[
 \ddet_{\Nh\Gamma} f \le \log_{\Nh\Z^3} f \; .
\]
In general strict inequality holds here. It would be interesting to try to evaluate \eqref{eq:9} for special polynomials $a \in \Z [y,z]$ and look for relations with arithmetic. 

\subsection{Expansiveness and the entropy of algebraic dynamical systems} \label{subsec:23}
The dynamical systems whose entropy is related to regulators in special cases are defined as follows. For a countable discrete group $\Gamma$ consider a countable $\Gamma$-module $M$. Its Pontrjagin dual $X = \hM$ is then a compact metrizable abelian group on which $\Gamma$ acts by topological group automorphisms. For $\Gamma = \Z^d$ the theory of such systems is explained in \cite{Sch}. See \cite{CL} and \cite{Li} for recent results concerning general groups $\Gamma$. 

The following case is the most relevant for us. For an element $f = \sum a_{\gamma} \gamma$ in $\Z \Gamma$ consider the $\Gamma$-module $M = \Z \Gamma / \Z \Gamma f$ with $\Gamma$ acting by left multiplication. Then $X_f = \hM$ is a subshift of the full shift $(\R / \Z)^{\Gamma} = \widehat{\Z\Gamma}$. It consists of all families $(x_{\gamma})_{\gamma \in \Gamma}$ with $x_{\gamma} \in \R / \Z$ such that for all $\gamma \in \Gamma$ the relation
\[
 \sum_{\gamma'} a_{\gamma'} x_{\gamma\gamma'} = 0
\]
holds in $\R / \Z$. The group $\Gamma$ acts on $x = (x_{\gamma}) \in X_f$ by the formula $(\gamma x)_{\gamma'} = x_{\gamma^{-1} \gamma'}$. It was known for a long time that for $f \in \Z [ \Z^d]$ the $\Z^d$-action on $X_f$ is expansive if and only if $\hf$ has no zeroes on $T^d$. In view of Wiener's theorem \ref{t2} this is equivalent to $f$ being invertible in $L^1 (\Z^d)$. This statement generalizes by a theorem of Choi \cite{C}, (see also \cite{DS} Theorem 3.2).

\begin{theorem} \label{t3}
 For a countable discrete group $\Gamma$ consider $f \in \Z \Gamma$. Then the action of $\Gamma$ on $X_f$ is expansive if and only if $f \in L^1 (\Gamma)^{\times}$.
\end{theorem}
The proof uses Kaplansky's theorem that in the algebra $\Nh\Gamma$ and hence in $L^1 (\Gamma)$ the relation $fg = 1$ implies $gf = 1$.

A nice generalization of this result was obtained by Chung and Li \cite{CL} theorem 3.1

\begin{theorem}\label{t3n}
 The action of $\Gamma$ on $X$ is expansive if and only if the $\Z \Gamma$-module $M = \hX$ is finitely generated and $L^1 (\Gamma) \otimes_{\Z\Gamma} M = 0$.
\end{theorem}

For $\Gamma = \Z^d$ expansiveness can be characterized as follows where condition 1) is part of \cite{Sch} Theorem 6.5 and condition 2) is \cite{Br} Theorem 5.2 or independently \cite{CL} Corollary 3.2.

\begin{theorem}\label{t3a}
 The $\Z^d$-action on $X$ is expansive if and only if $M$ is a finitely generated $R_d = \Z [ \Z^d]$-module and one of the following equivalent conditions holds:\\
1) For every associated prime ideal $\ep$ of the $R_d$-module $M$ the zero variety $V (\ep)$ of $\ep$ in $(\C^{\times})^d$ has empty intersection with $T^d$.\\ 
2) The module $M$ is $S_{\infty}$-torsion, where $S_{\infty} \subset R_d$ is the multiplicative system $S_{\infty} = R_d \cap L^1 (\Z^d)^{\times}$. 
\end{theorem}

Note that then $M$ has to be an $R_d$-torsion module.

For $f \in \Z \Gamma$ let $h (f)$ be the entropy of the $\Z\Gamma$-action on $X_f$. For $\Gamma = \Z^d$ it was determined in \cite{LSW} building on ideas of Sinai and Yuzvinskii who dealt with the case $d = 1$:

\begin{theorem} \label{t4}
 For non-zero $f \in \Z [\Z^d]$ we have the formula $h (f) = m (\hf) $.
\end{theorem}

After realizing that the Mahler measure is the logarithm of a Fuglede--Kadison determinant \eqref{eq:6} it was clear what the entropy $h (f)$ should be for general $\Gamma$. In the papers \cite{D2}, \cite{D3} the expansive case was treated under certain assumptions by two different methods. After overcoming all the remaining technical difficulties of the first approach in \cite{D2}, Li managed to prove the following result in \cite{Li}:

\begin{theorem} \label{t5}
 Let $\Gamma$ be a countable amenable discrete group and assume that $f \in \Z \Gamma$ is a unit in $\Nh\Gamma$. Then we have
\begin{equation} \label{eq:11}
 h (f) = \log \ddet_{\Nh\Gamma} f \; .
\end{equation}
\end{theorem}

In particular, formula \eqref{eq:11} holds for all expansive systems $X_f$. More generally, I think that \eqref{eq:11} should hold whenever $f$ is not a zero-divisor in $\R \Gamma$. This is true for $\Gamma = \Z^d$ and trivially for finite groups $\Gamma$. Work in progress by Hanfeng Li and Andreas Thom apparently settles the general case. Under suitable assumptions the relation \eqref{eq:11} also holds for Lewis Bowen's entropy if $\Gamma$ is not amenable, c.f. \cite{B}. In \cite{DS} Theorem 5.7 it was shown that for expansive systems $X_f$ the entropy can be calculated in terms of periodic points:

\begin{theorem} \label{t6}
 Let $\Gamma$ be a residually finite amenable group and assume that $f \in \Z \Gamma$ is a unit in $L^1 (\Gamma)$ i.e. that the $\Gamma$-action on $X_f$ is expansive. Then we have the formula
\[
 h (f) = h^{\per} (f) = \lim_{n\to\infty} \frac{1}{(\Gamma : \Gamma_n)} \log |\Fix_{\Gamma_n} (X_f)|
\]
for any choice of sequence $\Gamma_n \to e$.
\end{theorem}
 
This result will motivate the discussion of the $p$-adic case later.

For $\Gamma = \Z^d$ Br\"auer gave an interesting $K$-theoretical generalization of theorem \ref{t5}. Let $\Mh_{S_{\infty}} (R_d)$ be the category of finitely generated $R_d$-modules which are $S_{\infty}$-torsion. This is an exact subcategory of the category of all $R_d$-modules and one has a localization sequence \cite{Ba} IX Theorem 6.3 and Corollary 6.4
\[
 R^{\times}_d = K_1 (R_d) \longrightarrow K_1 (R_d [S^{-1}_{\infty}]) \xrightarrow{\delta} K_0 (\Mh_{S_{\infty}} (R_d)) \xrightarrow{\varepsilon} K_0 (R_d) \; .
\]
The map $\delta$ induces an isomorphism
\[
 \overline{\delta} : K_1 (R_d [S^{-1}_{\infty} ] ) / R^{\times}_d \silo K_0 (\Mh_{S_{\infty}} (R_d))
\]
whose inverse will be denoted by $cl_{\infty}$. For example
\[
 cl_{\infty} (R_d / R_d f) = [f] \mod R^{\times}_d \; .
\]
The composition obtained using \eqref{eq:3a}
\[
\ddet_{\Nh\Z^d} : K_1 (R_d [S^{-1}_{\infty}]) \longrightarrow K_1 (\Nh \Z^d) \xrightarrow{\ddet_{\Nh\Z^d}} \R^*_+
\]
is trivial on $R^{\times}_d = \pm \Z^d$. Using Yuzvinskii's addition formula for entropy and calculations of entropy in \cite{Sch}, Br\"auer obtained the following result \cite{Br} Theorem 5.8

\begin{theorem} \label{t6a}
 Consider an expansive algebraic $\Z^d$-action on a compact abelian group $X$. Then we have
\[
 h (X) = \log \ddet_{\Nh \Z^d} (cl_{\infty} (X)) \; .
\]
\end{theorem}

On the subgroup $SK_1 (R_d [S^{-1 }_{\infty}])$ the determinant $\ddet_{\Nh\Z^d}$ is trivial \cite{Br} Corollary 5.7. Using work by Dayton and real Bott periodicity, Br\"auer proves that for $d \ge 5$ the group $SK_1 (R_d [S^{-1}_{\infty}])$ is non-zero and that there are expansive algebraic actions of $\Z^d$ such that the $SK_1$-component $f_X$ of $cl_{\infty}$ is non-zero \cite{Br} Theorem 5.35 ff. It would be interesting to know if $f_X$ has a dynamical interpretation. 

I think that a version of theorem \ref{t6a} should hold for general $\Gamma$. For this the ideas in \cite{CL} section 3 should be helpful.
  
\subsection{The relation to the Deligne regulator} \label{subsec:24}
Consider the multiplicative algebraic torus $\G^d_m = \spec \Z [\Z^d]$. For an element $f \in \Z [\Z^d] = \Gamma (\G^d_m , \Oh)$ let $V (f)$ be the zero scheme in $\G^d_m$ and $U$ its complement. Thus $U = \spec \Z [\Z^d]_{f}$. Consider the regulator map from motivic cohomology to Deligne cohomology
\[
 r_{\Dh} : H^{d+1}_{\Mh} (U , \Q (d+1)) \longrightarrow H^{d+1}_{\Dh} (U_{\R} , \R (d+1)) \; .
\]
We may think of $H^{d+1}_{\Mh} (U , \Q (d+1))$ as a subquotient of the algebraic $K$-group of the localization $\Z [\Gamma]_{f}$ of the group ring $\Gamma = \Z^d$. By dimension reasons we have
\[
 H^{d+1}_{\Dh} (U_{\R} , \R (d+1)) = H^d_B (U_{\R} , \R (d)) \; .
\]
 If we assume that $\hf$ does not vanish in any point of $T^d \subset (\C^*)^d$, then $T^d$ with its standard orientation defines a $d$-dimensional homology class $[T^d]$ on $U (\C)$. The element $(2 \pi i)^{-d} [T^d]$ is $\overline{F}_{\infty}$-invariant where $F_{\infty}$ is complex conjugation on $U (\C)$. Hence it defines a class in $H^B_d (U_{\R} , \Z (-d))$. There is a natural pairing
\[
 \langle , \rangle : H^d_B (U_{\R} , \R (d)) \times H^B_{-d} (U_{\R} , \R (-d)) \longrightarrow \R \; .
\]
The standard generators $e_i$ of $\Z^d$ for $i = 1 , \ldots , d$ define elements of $\Z [\Z^d]^{\times} = \Gamma (\G^d_m , \Oh)^{\times}$ and hence of $\Gamma (U , \Oh)^{\times}$. Consider the cup-product
\[
 \{ f , e_1 , \ldots , e_d \} = f \cup e_1 \cup \ldots \cup e_d \in H^{d+1}_{\Mh} (U , \Q (d+1))
\]
of the $d+1$ elements $f , e_1 , \ldots , e_d$ in $H^1_{\Mh} (U , \Q (1)) = \Oh^{\times} (U) \otimes \Q$. 

A basic calculation with $\cup$-products in Deligne cohomology gives the following result \cite{D1} proposition 1.2.

\begin{prop} \label{t7}
 For $f \in \Z [\Z^d]$ with $\hf (z) \neq 0$ for all $z \in T^d$ we have
\begin{equation} \label{eq:12}
 \langle r_{\Dh} \{ f , e_1 , \ldots , e_d \} , (2 \pi i)^{-d} [T^d] \rangle = m (\hf) \; .
\end{equation}
\end{prop}

According to the preceeding sections, we can express this statement differently and without using the Fourier transform $\hf$ of $f$:

\begin{cor} \label{t8}
For $f \in \Z [\Z^d]$ with $f \in L^1 (\Z^d)^{\times}$ or equivalently with the $\Z^d$-action on $X_f$ being expansive, we have
\[
 \langle r_{\Dh} \{ f , e_1 , \ldots , e_d \} , (2 \pi i)^{-d} [T^d] \rangle  = h (f) = \log \ddet_{\Nh\Z^d} (f) \; .
\]
\end{cor}

Neither the relation of the regulator pairing with entropy nor the relation with the determinant on $\Nh \Z^d$ is conceptually understood. It would be interesting to generalize the left hand side to some non-commutative groups $\Gamma$ e.g. polycyclic ones. The starting point could be the algebraic $K$-theory of the Ore localization $\Z \Gamma_{f}$. This would give integrals like \eqref{eq:9} an interpretation as non-commutative regulator values. The expression on the left of equation \eqref{eq:12} may also be interpreted as a Deligne period of a mixed motive, \cite{D1}. Equation \eqref{eq:12} and its refinements embed $m (\hf)$ into the context of the Beilinson conjectures. See \cite{D1}, \cite{La}, \cite{RV}.

There is a trivial analogue of proposition \ref{t7} for the group $(\Z / n)^d$ instead of $\Z^d$. Consider the group scheme $\mu^d_n = \spec \Z [(\Z / n)^d]$. For an element $F \in \Z [(\Z/n)^d] = \Gamma (\mu^d_n , \Oh)$ let $V (F)$ be the zero scheme of $F$ in $\mu^d_n$ and $U_n$ its complement. Again we have a regulator map
\[
 r_{\Dh} : H^1_{\Mh} (U_n, \Q (1)) \longrightarrow H^1_{\Dh} (U_{n\R} , \R (1)) = H^0_B (U_{n\R} , \R) \; .
\]
If $F (\zeta) \neq 0$ for all $\zeta \in \mu^d_n (\C)$ then $U_{n\R} = \mu^d_{n\R}$ and we find
\[
 H^0_B (U_{n\R} , \R) = \Big( \prod_{\zeta \in \mu^d_n (\C)} \R \Big)^+ \; .
\]
Here an element $(\lambda_{\zeta})$ is in the $+$ part if $\lambda_{\ozeta} = \lambda_{\zeta}$ for all $\zeta \in \mu^d_n (\C)$. Using the canonical pairing
\[
 \langle , \rangle : H^0_B (\mu^d_n (\C) , \R) \times H^B_0 (\mu^d_n (\C) , \R) \longrightarrow \R
\]
and formula \eqref{eq:5} we immediately get

\begin{prop} \label{t9}
 For $F \in \Z [(\Z / n)^d]$ with $F (\zeta) \neq 0$ for all $\zeta \in \mu^d_n (\C)$ or equivalently with $F \in \C [(\Z / n)^d]^{\times}$ we have
\begin{align*}
 \langle r_{\Dh} (F) , n^{-d} (1, \ldots , 1) \rangle & = n^{-d} \sum_{\zeta \in \mu^d_n (\C)} \log |F (\zeta)| \\
& = h (F) = \log \ddet_{\Nh (\Z / n)^d} (F) \; .
\end{align*}
\end{prop}

Starting with an element $f \in \Z [\Z^d] = \Gamma (\G^d_m , \Oh)$ we may apply this proposition to $F = f^{(n)}$ which may be viewed as the restriction of the regular function $f$ on $\G^d_m$ to $\mu^d_n \subset \G^d_m$. Thus $U_n = U \cap \mu^d_n \subset U$ and equation \eqref{eq:7} gives the following limit formula if we assume that $\hf$ is non-zero in any point of $T^d$:
\begin{equation} \label{eq:13}
 \langle r_{\Dh} \{ f , e_1 , \ldots , e_d \} , (2 \pi i)^{-d} [T^d] \rangle = \lim_{n\to\infty} \langle r_{\Dh} (f \, |_{\mu^d_n}) , n^{-d} (1, \ldots , 1) \rangle \; .
\end{equation}
It would be interesting to find an approximation property of regulators with respect to dense subschemes generalizing this statement. In this regard the evaluations on $(2 \pi i)^{-d} [T^d]$ and $n^{-d} (1 , \ldots , 1)$ in formula \eqref{eq:13} should be viewed as normalized traces on cohomology.  

\section{The non-archimedian case} \label{sec:3}
\subsection{$p$-adic expansiveness and $p$-adic entropy for algebraic dynamical systems} \label{subsec:31}

Let $p$ be a prime number. In the following sections we will describe a $p$-adic analogue of the previous theory. Only algebraic dynamical systems will be considered because the $p$-adic theory at present works only in such a rigid situation. Thus, as before let $\Gamma$ be a countable discrete group acting on a compact metrizable topological group $X$ by topological automorphisms. The usual definitions of the topological or  measure theoretic entropy do not seem to have a good $p$-adic analogue. However it turns out that a $p$-adic analogue of the logarithmic growth rate of the number of  periodic points exists for a large class of examples. Consider the branch of the $p$-adic logarithm $\log_p : \Q^{\times}_p \to \Z_p$ normalized by $\log_p p = 0$. If $\Gamma$ is residually finite and if the limit
\[ 
 h^{\per}_p = \lim_{n\to \infty} \frac{1}{(\Gamma : \Gamma_n)} \log_p |\Fix_{\Gamma_n} (X)|
\]
exists in $\Q_p$ for all sequences $\Gamma_n \to e$, we will call it the $p$-adic periodic entropy of the $\Gamma$-action, \cite{D3}. This notion turns out to be reasonable for a large class of systems of the form $X = X_f$ with $f \in \Z \Gamma$. We will calculate it in terms of a $p$-adic analogue of the Fuglede--Kadison determinant of $f$. For $\Gamma = \Z^d$ we will also relate it to values of the syntomic regulator map.

We need some objects from $p$-adic functional analysis. The definition of $L^1 (\Gamma)$ has an obvious $p$-adic analogue but due to the ultrametric inequality it is more natural to look at the following bigger algebra which has no archimedian analogue. Let $c_0 (\Gamma)$ be the $\Q_p$-Banach algebra of all formal series $\sum x_{\gamma} \gamma$ with $x_{\gamma} \in \Q_p$ and $|x_{\gamma}|_p \to 0$ for $\gamma \to \infty$, equipped with the norm $\| x \| = \max_{\gamma} |x_{\gamma}|_p$. The elements $x$ with $\| x \| \le 1$ form the $\Z_p$-Banach algebra $c_0 (\Gamma , \Z_p)$ where in addition $x_{\gamma} \in \Z_p$ for all $\gamma$. For $\Gamma = \Z^d$ we write the natural isomorphism with the Tate algebra as
\[
 c_0 (\Z^d) \silo \Q_p \langle z^{\pm 1}_1 , \ldots , z^{\pm 1}_d \rangle \; , \; f \mapsto \hf \; .
\]
Here $z_1 , \ldots , z_d$ are the coordinates on $(\C^{\times}_p)^d$. By $T^d_p$ we denote the $p$-adic $d$-torus of points $z \in (\C^{\times}_p)^d$ with $|z_i|_p = 1$ for all $1 \le i \le d$. The following well known fact from $p$-adic analysis may be viewed as an analogue of Wiener's theorem \ref{t2}:

\begin{prop} \label{t12}
 For $f \in c_0 (\Z^d)$ the following conditions are equivalent
\begin{compactenum}
 \item $\hf (z) \neq 0$ for all $z \in T^d_p$
\item $f$ is a unit in $c_0 (\Z^d)$.
\end{compactenum}
\end{prop}

Note that the reduction $\mod p$ of $c_0 (\Gamma , \Z_p)$ is $\F_p [\Gamma]$. Moreover an element $f$ of $c_0 (\Gamma, \Z_p)$ is a unit if and only if its reduction is a unit in $\F_p [\Gamma]$. We have an exact sequence of groups
\begin{equation} \label{eq:14}
 1 \longrightarrow U^1 \longrightarrow c_0 (\Gamma , \Z_p)^{\times} \longrightarrow \F_p [\Gamma]^{\times} \longrightarrow 1
\end{equation}
where $U^1 = 1 + p c_0 (\Gamma , \Z_p)$ is the subgroup of $1$-units in $c_0 (\Gamma , \Z_p)$. Note that for any homomorphism of groups $\varphi : \Gamma \to \Gamma'$ there is an induced continuous algebra homomorphism
\[
 \varphi_* : c_0 (\Gamma) \longrightarrow c_0 (\Gamma') \quad \text{where} \quad \varphi_* ({\textstyle\sum} a_{\gamma} \gamma) = {\textstyle\sum} a_{\gamma} \varphi (\gamma) \; .
\]
I conjecture that in $c_0 (\Gamma)$ the relation $fg = 1$ implies that $gf = 1$. This is true for residually finite groups $\Gamma$ because of the natural injection
\[
 c_0 (\Gamma) \hookrightarrow \prod^{\infty}_{n=1} c_0 (\Gamma / \Gamma_n)
\]
into a product of finite dimensional $\Q_p$-algebras. Here the $\Gamma_n$'s are cofinite normal subgroups with $\bigcap^{\infty}_{n=1} \Gamma_n = \{ e \}$. 

An abelian group $X$ is said to have bounded $p$-torsion if there is some integer $i_0 \ge 0$ with
\[
 \Ker (p^i : X \to X) = \Ker (p^{i_0} : X \to X) \quad \text{for all} \; i \ge i_0 \; .
\]
We have the following result:

\begin{theorem} \label{t14}
 For $f \in \Z \Gamma$ the following conditions are equivalent\\
a) The group $X_f$ has bounded $p$-torsion\\
b) There is an element $g \in c_0 (\Gamma)$ with $gf = 1$\\
c) $c_0 (\Gamma) \otimes_{\Z\Gamma} M_f = 0$ where $M_f = \hX_f = \Z \Gamma / \Z \Gamma f$.

In this case $\Fix_N (X_f)$ is finite for any cofinite normal subgroup $N \vartriangleleft \Gamma$.

If $\Gamma$ is residually finite, condition a)--c) are equivalent to\\
d) The element $f$ is a unit in $c_0 (\Gamma)$.
\end{theorem}

For obvious reasons we will call the $\Gamma$-action on $X_f$ $p$-adically expansive if either condition a)--c) is satsified. Note that $M_f$ is a finitely generated $\Z\Gamma$-module. Condition a) was introduced by Br\"auer \cite{Br} for algebraic $\Gamma = \Z^d$-actions. 

\begin{proof}
 The isomorphism
\[
 c_0 (\Gamma) \otimes_{\Z\Gamma} M_f = c_0 (\Gamma) / c_0 (\Gamma) f 
\]
shows that b) and c) are equivalent. For a $\Z \Gamma$-module $M$ the Pontrjagin dual $X = \hM$ has bounded $p$-torsion if and only if the sequence
\[
 M \supset pM \supset p^2 M \supset \ldots 
\]
becomes stationary i.e. for some $i_0 \ge 0$ we have $p^{i_0} M = p^i M$ for $i \ge i_0$. Thus condition a) says that
\[
 p^{i_0} \Z \Gamma + \Z \Gamma f = p^i \Z \Gamma ´+ \Z \Gamma f \quad \text{for all} \; i \ge i_0 \; .
\]
Equivalently for $i \ge i_0$ there are elements $h_i , q_i \in \Z \Gamma$ with
\begin{equation} \label{eq:16}
 p^{i_0} = p^i h_i + q_i f \quad \text{for} \; i \ge i_0 \; .
\end{equation}
Setting $i_1 = i_0 + 1$ we get
\[
 q_{i_1} f = p^{i_0} (1 - p h_{i_1}) \; .
\]
The element $1 - ph_{i_1}$ is a $1$-unit in $c_0 (\Gamma , \Z_p)$ and hence $u = q_{i_1} f$ is in $c_0 (\Gamma)^{\times}$. Hence $g = u^{-1} q_{i_1} \in c_0 (\Gamma)$ satisfies $gf = 1$ and hence b) holds.

Conversely assume that we have $gf = 1$ for some $g \in c_0 (\Gamma)$. There is an integer $i_0 \ge 0$ with $q = p^{i_0} g \in c_0 (\Gamma , \Z_p)$. For each $i \ge i_0$ we may write $q$ in the form $q = p^i s_i + q_i$ where $s_i \in c_0 (\Gamma , \Z_p)$ and $q_i \in \Z \Gamma$. Thus, setting $h_i = s_i f$ we get
\[
 p^{i_0} = qf = p^i h_i + q_i f \; .
\]
Since $q_if$ is in $\Z \Gamma$ the element $h_i$ is in $c_0 (\Gamma , \Z_p) \cap p^{-i} \Z \Gamma = \Z \Gamma$ as well. Thus equation \eqref{eq:16} holds for all $i \ge i_0$ and a) is satisfied.

The Pontrjagin dual of $\Fix_N (X_f)$ is the $N$-cofix module of $\Z \Gamma / \Z \Gamma f$ i.e. the module $\Z [\oGamma] / \Z \oGamma \bar{f}$ where $\oGamma = \Gamma / N$ and $\of$ is the image of $f$ in $\Z\oGamma$. The relation $gf = 1$ in $c_0 (\Gamma)$ implies the relation $\og \of = 1$ in $c_0 (\oGamma) = \Q_p [\oGamma]$. \\
This implies that $(\Z \oGamma / \Z \oGamma \bar{f}) \otimes_{\Z} \Q_p = 0$. Hence the finitely generated abelian group $\Z \oGamma / \Z \oGamma \bar{f}$ is torsion and therefore finite. Thus $\Fix_N (X_f)$ is finite as well. Finally, it was mentioned above that condition b) is equivalent to d) if $\Gamma$ is residually finite.
\end{proof}

For $\Gamma = \Z^d$-actions, Br\"auer \cite{Br} Propositions 4.19 and 4.22 has shown the following result similar to theorem \ref{t3a}:

\begin{theorem} \label{t15}
 Consider an algebraic $\Z^d$-action on a compact abelian group $X$ for which the $R_d$-module $M = \hX$ is finitely generated. Then the following conditions are equivalent:\\
a) The group $X$ has bounded $p$-torsion\\
b) For every associated prime ideal $\ep$ of the $R_d$-module $M$ the zero variety $V_p (\ep)$ of $\ep$ in $(\C^{\times}_p)^d$ has empty intersection with $T^d_p$.\\
c) The module $M$ is $S_p$-torsion, where $S_p \subset R_d$ is the multiplicative system $S_p = R_d \cap c_0 (\Z^d)^{\times}$.
\end{theorem}

Note that $M$ then has to be an $R_d$-torsion module.

A $\Z^d$-algebraic system will be called $p$-adically expansive if either condition in the theorem is satisfied. This is compatible with the case $X_f$ above.

\begin{remark} \label{t16}
 \rm For general groups $\Gamma$ the $p$-adically expansive algebraic dynamical systems $X$ should be the ones with $M = \hX$ finitely generated as a $\Z\Gamma$-module and $c_0 (\Gamma) \otimes_{\Z\Gamma} M = 0$. However we want a non-functional analytic description of this property. I have not investigated whether bounded $p$-torsion of $X$ is the right condition also in general.
\end{remark}

We end the discussion of $p$-adic expansiveness with the following fact:

\begin{prop} \label{t17}
 If $X$ has bounded $p$-torsion and $M = \hX$ is a Noetherian $\Z\Gamma$-module, then $M$ is a torsion $\Z\Gamma$-module.
\end{prop}

\begin{proof}
 By assumption there is some $i_0 \ge 0$ with $p^{i_0} M = p^i M$ for all $i \ge i_0$. Assume there is a non-torsion element $m_0 \in M$. Choose elements $m_{\nu} \in M$ with $p^{2 i_0} m_{\nu} = p^{i_0} m_{\nu-1}$ for $\nu \ge 1$ and set $n_{\nu} = p^{i_0} m_{\nu}$. Thus we have $p^{i_0} n_{\nu} = n_{\nu-1}$ for $\nu \ge 1$. The submodules $N_{\nu} = \Z \Gamma n_{\nu}$ of $M$ are isomorphic to $\Z \Gamma$ and form a strictly increasing sequence contrary to $M$ being Noetherian by assumption.
\end{proof}

\subsection{$p$-adic determinants and their relation to $p$-adic entropy} \label{subsec:32}
The von~Neumann algebra $\Nh\Gamma$ has many idempotents given by orthogonal projections to (right) $\Gamma$-invariant closed subspaces of $L^2 (\Gamma)$. In the case of $\Nh\Z^d = L^{\infty} (T^d)$ for example they are the characteristic functions of measurable subsets of $T^d$. Evaluating the trace $\tau_{\Gamma}$ on idempotents defines the real valued non-negative ``continuous dimension'' of the corresponding subspace. It would be interesting to have a $p$-adic analogue of this theory of von~Neumann where the dimensions would be $p$-adic numbers. However in the $p$-adic setting no analogue of $\Nh\Gamma$ is known. The algebras $c_0 (\Gamma)$ are too small. For $\Gamma = \Z^d$ for example $c_0 (\Z^d) = \Q_p \langle z^{\pm 1}_1 , \ldots , z^{\pm 1}_d \rangle$ has no non-trivial idempotents just like $L^1 (\Z^d)$. On the other hand $c_0 (\Z^d)$ is even an integral domain, unlike $L^1 (\Z^d)$ which has very many zero divisors. Nevertheless we have seen in the previous section that for our purposes $c_0 (\Gamma)$ is a reasonable $p$-adic substitute of $L^1 (\Gamma)$. As we will see it is possible to define a $p$-adic analogue of $\log \ddet_{\Nh\Gamma}$ (not of $\ddet_{\Nh\Gamma}$!) on $c_0 (\Gamma)^{\times}$. For $\Gamma = \Z^d$ this is easy in terms of the $p$-adic Snirelman integral. For non-commutative $\Gamma$ some deep facts from algebraic $K$-theory are needed. They replace the functional calculus which is used in the definition of $\ddet_{\Nh\Gamma}$.

We proceed as follows. For $n \ge 1$ consider the $\Z_p$-Banach algebra $A_n = M_n (c_0 (\Gamma , \Z_p))$ with norm $\| (a_{ij}) \| = \max_{i,j} \| a_{ij} \|$. We have a trace map
\[
 \tau_{\Gamma} : c_0 (\Gamma , \Z_p) \longrightarrow \Z_p \quad \text{defined by} \; \tau_{\Gamma} (f) = a_e
\; \text{if} \; f = {\textstyle \sum} a_{\gamma} \gamma \; . 
\]
The composition
\[
 \tau_{\Gamma} : A_n \xrightarrow{\tr} c_0 (\Gamma , \Z_p) \xrightarrow{\tau_{\Gamma}} \Z_p
\]
vanishes on commutators $[u,v] = uv - vu$ in $A_n$. Using the $p$-adic Baker--Campbell--Hausdorff formula and an additional argument for $p = 2$ one shows that on the $1$-units $U^1_n = 1 + pA_n$ of $A^{\times}_n$ the map
\begin{equation} \label{eq:17}
 \log_p \ddet_{\Gamma} := \tau_{\Gamma} \verk \log : U^1_n \longrightarrow \Z_p
\end{equation}
is a homomorphism, c.f. \cite{D3} Theorem 4.1. Here $\log : U^1_n \to A_n$ is defined by the $p$-adically convergent power series
\[
 \log u = - \sum^{\infty}_{\nu=1} \frac{(1-u)^{\nu}}{\nu} \; .
\]
The maps $\log_p \ddet_{\Gamma}$ organize into a homomorphism
\[
 \log_p \ddet_{\Gamma} : 1 + p M_{\infty} (c_0 (\Gamma , \Z_p)) \longrightarrow \Z_p \; .
\]
The exact sequence
\[
 1 \longrightarrow 1+p M_{\infty} (c_0 (\Gamma , \Z_p)) \longrightarrow \GL_{\infty} (c_0 (\Gamma , \Z_p)) \longrightarrow \GL_{\infty} (\F_p [\Gamma]) \longrightarrow 1
\]
gives rise to the exact sequence:
\[
 1 \longrightarrow \Kh \longrightarrow K_1 (c_0 (\Gamma , \Z_p)) \longrightarrow K_1 (\F_p [\Gamma]) \longrightarrow 1 \; .
\]
Here $\Kh$ is the quotient of $U^1_{\infty} = 1 + p M_{\infty} (c_0 (\Gamma , \Z_p))$ by $U^1_{\infty} \cap E_{\infty} (c_0 (\Gamma , \Z_p))$ where $E_{\infty} (R)$ denotes the elementary matrices in $\GL_{\infty} (R)$. If $\Gamma$ is residually finite it can be shown that $\log_p \ddet_{\Gamma}$ factors over $\Kh$, c.f. \cite{D3} Theorem 5.1. This should be true in general. Next, we extend $\log_p \ddet_{\Gamma} : \Kh \to \Z_p$ by the unique divisibility of $\Q_p$ to a homomorphism on $K_1 (c_0 (\Gamma , \Z_p))$:

Let $\langle \Gamma \rangle$ be the image of $\Gamma$ under the canonical map $\F_p [\Gamma]^{\times} \to K_1 (\F_p [\Gamma])$ and define the Whitehead group of $\Gamma$ over $\F_p$ as the quotient:
\[
 Wh^{\F_p} (\Gamma) = K_1 (\F_p [\Gamma]) / \langle \Gamma \rangle \; .
\]
The following result is now clear:

\begin{theorem} \label{t18}
 Let $\Gamma$ be a residually finite group such that $Wh^{\F_p} (\Gamma)$ is torsion. Then there is a unique homomorphism
\[
 \log_p \ddet_{\Gamma} : K_1 (c_0 (\Gamma , \Z_p)) \longrightarrow \Q_p
\]
with the following properties:\\
{\bf a} For every $n \ge 1$ the composition
\[
 U^1_n \longrightarrow A^{\times}_n \longrightarrow K_1 (c_0 (\Gamma , \Z_p)) \xrightarrow{\log_p \ddet_{\Gamma}} \Q_p
\]
is the map \eqref{eq:17}.\\
{\bf b} On the image of $\Gamma$ in $K_1 (c_0 (\Gamma, \Z_p))$ the map $\log_p \ddet_{\Gamma}$ vanishes.
\end{theorem}

Cases where $Wh^{\F_p} (\Gamma)$ is known to be torsion include torsion-free elementary amenable groups $\Gamma$, c.f. \cite{FL} and word hyperbolic groups \cite{BLR}.

We define the desired logarithmic determinant on $c_0 (\Gamma , \Z_p)^{\times}$ as the composition
\begin{equation} \label{eq:18}
 \log_p \ddet_{\Gamma} : c_0 (\Gamma , \Z_p)^{\times} \longrightarrow K_1 (c_0 (\Gamma , \Z_p)) \xrightarrow{\log_p \ddet_{\Gamma}} \Q_p \; .
\end{equation}
If $\F_p [\Gamma]$ has no zero-divisors it is easy to see that
\[
 c_0 (\Gamma)^{\times} = p^{\Z} c_o (\Gamma , \Z_p)^{\times} \quad \text{and} \quad p^{\Z} \cap c_0 (\Gamma , \Z_p)^{\times} = 1 \; .
\]
In this case, setting $\log_p \ddet_{\Gamma} (p) = 0$ we obtain a homomorphism extending \eqref{eq:18}
\[
 \log_p \ddet_{\Gamma} : c_0 (\Gamma)^{\times} \longrightarrow \Q_p \; .
\]
If $\Gamma$ is torsion-free and elementary amenable then $\F_p [\Gamma]$ has no zero-divisors \cite{KLM} Theorem 1.4. A $p$-adic analogue of the previous calculation of the classical entropy is given by the following result, \cite{D3}, propositions 3.1 and 5.5.

\begin{theorem} \label{t19}
 Assume that the residually finite group $\Gamma$ is elementary amenable and torsion free. Let $f$ be an element of $\Z\Gamma$ such that $X_f$ is $p$-adically expansive i.e. $X_f$ has bounded $p$-torsion or equivalently $f \in c_0 (\Gamma)^{\times}$. Then the $p$-adic periodic entropy $h^{\per}_p (f)$ of the $\Gamma$-action on $X_f$ exists and we have
\[
 h^{\per}_p (f) = \log_p \ddet_{\Gamma} f \; .
\]
\end{theorem}

We now make this explicit in the case $\Gamma = \Z^d$ where the outcome will be a $p$-adic Mahler measure. Recall that the Snirelman integral of a continuous function $\varphi : T^d_p \to \C_p$ is defined by the following limit if it exists, \cite{BD} \S\,1
\[
 \int_{T^d_p} \varphi = \lim_{n\to \infty \atop (n,p) = 1} n^{-d} \sum_{\zeta \in \mu^d_N} \varphi (\zeta) \; .
\]
Here $\mu_N$ is the group of $N$-th roots of unity in $\oQ^{\times}_p$. It can be shown that for a function $\hf \in \C_p \langle z^{\pm 1}_1 , \ldots , z^{\pm 1}_d \rangle$ which does not vanish in any point of $T^d_p$ the integral
\[
 m_p (\hf) := \int_{T^d_p} \log_p \hf
\]
exists. It is called the $p$-adic (logarithmic) Mahler measure. Here $\log_p : \C^{\times}_p \to \C_p$ is the $p$-adic logarithm normalized by $\log_p (p) = 0$. For a Laurent polynomial $\hf (z) = a_m z^m + \ldots + a_r z^r$ in $\C_p [z , z^{-1}]$ with $a_m , a_r$ non-zero, we have by \cite{BD} Proposition 1.5
\begin{align}
 m_p (\hf) & = \log_p a_r - \sum_{0 < |\alpha|_p < 1} \log_p \alpha \label{eq:19} \\
& = \log_p a_m + \sum_{|\alpha|_p > 1} \log_p \alpha \nonumber \; . 
\end{align}
Here $\alpha$ runs over the zeroes of $\hf$ with their multiplicities. This formula corresponds to formula \eqref{eq:8} for the ordinary Mahler measure. Setting $\tau (\varphi) = \int_{T^d_p} \varphi$ there is a commutative diagram
\[
 \xymatrix{
c_0 (\Z^d) \ar[r]^{\tau_{\Gamma}} \ar[d]^{\wr}_{\wedge} & \Q_p \ar@{=}[d] \\
\Q_p \langle z^{\pm 1}_1 , \ldots , z^{\pm 1}_d \rangle \ar[r]^-{\tau} & \Q_p \; .
}
\]
Using this it is shown in \cite{D3} Remark 5.2 that for $f \in c_0 (\Z^d)$ we have:
\begin{equation} \label{eq:20}
 \log_p \ddet_{\Z^d} f = m_p (\hf) \; .
\end{equation}
This corresponds to formulas \eqref{eq:6} and \eqref{eq:7}.

We finish this section with some ideas of Br\"auer for general $p$-adically expansive actions in the case $\Gamma = \Z^d$, c.f. \cite{Br} 4.3. Let $\Mh_{S_p} (R_d)$ be the category of finitely generated $R_d$-modules which are $S_p = R_d \cap c_0 (\Z^d)^{\times}$-torsion. As in section \ref{subsec:23} the exact localization sequence
\[
 R^{\times}_d = K_1 (R_d) \longrightarrow K_1 (R_d [S^{-1}_p]) \xrightarrow{\delta} K_0 (\Mh_{S_p} (R_d)) \xrightarrow{\varepsilon} K_0 (R_d)
\]
leads to an isomorphism:
\[
 \overline{\delta} : K_1 (R_d [S^{-1}_p]) / R^{\times}_d \silo K_0 (\Mh_{S_p} (R_d)) \; .
\]
Its inverse is denoted by $cl_p$. For example
\[
 cl_p (R_d / R_d f) = [f] \mod R^{\times}_d \; .
\]
The map $\log_p \ddet_{\Z^d}$ on $K_1 (c_0 (\Z^d , \Z_p))$ has a natural extension to $K_1 (c_0 (\Z^d))$ because it is trivial on $SK_1 (c_0 (\Z^d , \Z_p))$. Consider the composition
\[
 \log_p \ddet_{\Z^d} : K_1 (R_d [S^{-1}_p]) \longrightarrow K_1 (c_0 (\Z^d)) \xrightarrow{\log_p \ddet_{\Z^d}} \Q_p \; .
\]
It is trivial on $R^{\times}_d = \pm \Z^d$ and hence the $p$-adic number $\log_p \ddet_{\Z^d} (cl (X))$ is defined for every $p$-adically expansive $\Z^d$-action $X$. Comparing with theorem \ref{t6a} we see that it should equal the ``$p$-adic entropy'' of $X$. For $X = X_f$ this is the case as we have seen before. In general however, $h^{\per}_p (X)$ as we have defined it may not exist as Br\"auer observes in \cite{Br} Example 7.1. It is an open problem even for $\Gamma = \Z^d$ to define dynamically a better notion of $p$-adic entropy which applies to all $p$-adically expansive $\Gamma$-actions $X$ and equals $\log_p \ddet_{\Z^d} (cl (X))$.

\subsection{The relation to the syntomic regulator} \label{subset:33}
We take up the notations from section \ref{subsec:24}. Thus for an element $f \in \Z [\Z^d] = \Gamma (\G^d_m , \Oh)$ let $U = \G^d_m \setminus V (f)$ be the complement of the zero scheme $V (f)$ of $f$ in $\G^d_m$. Consider the regulator map into syntomic cohomology \cite{Be}, \cite{BD} 2.3
\[
 r_{\syn} : H^{d+1}_{\Mh} (U , \Q (d+1)) \longrightarrow H^{d+1}_{\syn} (U_{\Z_p} , d+1) \; .
\]
There is a canonical isomorphism \cite{BD} (7)
\[
 H^{d+1}_{\syn} (U_{\Z_p} , d+1) \silo H^d_{\rig} (U_{\F_p} / \Q_p)
\]
with Berthelot's rigid cohomology \cite{Ber}. The cohomology classes in $H^d_{\rig} (U_{\F_p} / \Q_p)$ are represented by overconvergent $d$-forms $\omega$. Writing
\[
 \omega = f \frac{d z_1}{z_1} \wedge \ldots \wedge \frac{d z_d}{z_d}
\]
the function $f$ restricts to a rigid analytic function $T^d_p$ and we set
\[
 \langle [\omega] , [T^d_p] \rangle := \int_{T^d_p} f \; .
\]
This gives a well defined $\Q_p$-valued linear form $[T^d_p]$ on $H^d_{\rig} (U_{\F_p} / \Q_p)$ which we may view as a rigid homology class in $H^{\rig}_d (U_{\F_p} / \Q_p)$.

Defining the symbol $\{ f , e_1 , \ldots , e_d \}$ in $H^{d+1}_{\Mh} (U , \Q (d+1))$ as in section \ref{subsec:24} the following result holds, \cite{BD} Fact 2.4.

\begin{prop} \label{t20}
 For $f \in \Z [\Z^d]$ with $\hf (z) \neq 0$ for all $z \in T^d_p$ we have
\[
 \langle r_{\syn} \{ f , e_1 , \ldots , e_d \} , [T^d_p] \rangle = m_p (\hf) \; .
\]
\end{prop}

We can express this differently in terms of entropy and determinants as follows:

\begin{cor} \label{t21}
 For $f \in \Z [\Z^d]$ with $f \in c_0 (\Z^d)^{\times}$ or equivalently with the $\Z^d$-action on $X_f$ being $p$-adically expansive, we have
\[
 \langle r_{\syn} \{ f , e_1 , \ldots , e_d \} , [T^d_p] \rangle = h^{\per}_p (f) = \log_p \ddet_{\Z^d} (f) \; .
\]
\end{cor}
As in the archimedian case it would be interesting if this formula could be generalized to some non-commutative groups $\Gamma$. 

Apart from the $p$-adic Mahler measure which we just discussed, there is another more global one which depends on the choice of embeddings
\[
 \C \overset{\sigma_{\infty}}{\hookleftarrow} \oQ \overset{\sigma_p}{\hookrightarrow} \oQ_p
\]
and takes values in $B_{dR}$. Assume that $V (\hf) \cap T^d = \emptyset$ and a further technical condition on $\hf$ c.f. \cite{BD} Assumption 3.1. Fixing $\sigma = (\sigma_{\infty} , \sigma_p)$ one can use Falting's or Tsuji's comparison theorems in $p$-adic cohomology to transfer the homology class $(2 \pi i)^{-d} [T^d]$ in the Betti homology of $U_{\R}$ to a $B_{dR}$-valued rigid homology class $c_{\sigma}$. We set
\[
 m_{p,\sigma} (\hf) := \langle r_{\syn} \{ f , z_1 , \ldots , z_d \} , c_{\sigma} \rangle \in F^{-d} B_{dR} \subset B_{dR} \; .
\]
The Laurent polynomials $\hf$ in $\Z [z^{\pm 1}]$ to which this applies have the form
\begin{equation} \label{eq:21}
 \hf (z) = z^N a \prod_{\alpha} (z-\alpha) \quad \text{where} \; \alpha \in \oQ
\end{equation}
with $N \in \Z$ and with $|a|_p = 1 = |\sigma_p (\alpha)|_p$ and $|\alpha|_{\infty} := |\sigma_{\infty} (\alpha)| \neq 1$ for all $\alpha$. Writing $\hf$ in the form
\[
 \hf (z) = a_m z^m + \ldots + a_r z^r
\]
with $a_m = a$ and $a_r \neq 0$, we have $|a_r|_p = 1$ and
\begin{align*}
m_{p,\sigma} (\hf) & = \log_p a_r - \sum_{0 < |\alpha|_{\infty} < 1} \log_p (\sigma_p (\alpha)) \\
& = \log_p a_m + \sum_{|\alpha|_{\infty} > 1} \log_p (\sigma_p (\alpha)) \; .
\end{align*}
There are two-variable polynomials $\hf$ for which $m_{p,\sigma} (\hf)$ can be related to values of $p$-adic $L$-series of elliptic curves, c.f. \cite{BD} Theorem 3.11. Thus $m_{p,\sigma} (\hf)$ seems to be an interesting quantity. I do not know if there is a definition of $p$-adic entropy depending on $\sigma$ which for the $\Z^d$-action on $X_f$ gives $m_{p,\sigma} (\hf)$ under suitable conditions. Neither do I know a definition of a $p$-adic logarithmic determinant depending on $\sigma$ whose value on suitable $f$'s gives $m_{p,\sigma} (\hf)$. 


\begin{thebibliography}{KLM88}

\bibitem[B]{B}
Lewis Bowen.
\newblock Entropy for expansive algebraic actions of residually finite groups.
\newblock {\em Ergodic Theory Dynam. Systems}, 31(3):703--718, 2011.

\bibitem[BD]{BD}
Amnon Besser and Christopher Deninger.
\newblock {$p$}-adic {M}ahler measures.
\newblock {\em J. Reine Angew. Math.}, 517:19--50, 1999.

\bibitem[BLR]{BLR}
Arthur Bartels, Wolfgang L{\"u}ck, and Holger Reich.
\newblock On the {F}arrell-{J}ones conjecture and its applications.
\newblock {\em J. Topol.}, 1(1):57--86, 2008.

\bibitem[Ba]{Ba}
Hyman Bass.
\newblock {\em Algebraic {$K$}-theory}.
\newblock W. A. Benjamin, Inc., New York-Amsterdam, 1968.

\bibitem[Be]{Be}
Amnon Besser.
\newblock Syntomic regulators and {$p$}-adic integration. {I}. {R}igid syntomic
  regulators.
\newblock In {\em Proceedings of the {C}onference on {$p$}-adic {A}spects of
  the {T}heory of {A}utomorphic {R}epresentations ({J}erusalem, 1998)}, volume
  120, pages 291--334, 2000.

\bibitem[Ber]{Ber}
Pierre Berthelot.
\newblock Finitude et puret\'e cohomologique en cohomologie rigide.
\newblock {\em Invent. Math.}, 128(2):329--377, 1997.
\newblock With an appendix in English by Aise Johan de Jong.

\bibitem[Bo]{Bo}
David~W. Boyd.
\newblock Mahler's measure and special values of {$L$}-functions.
\newblock {\em Experiment. Math.}, 7(1):37--82, 1998.

\bibitem[Br]{Br}
Jonas Br\"auer.
\newblock Entropies of algebraic $\mathbb{Z}^d$-actions and {$K$}-theory, 2010.
\newblock Dissertation, M\"unster.

\bibitem[C]{C}
Yemon Choi.
\newblock Injective convolution operators on {$\ell^\infty(\Gamma)$} are
  surjective.
\newblock {\em Canad. Math. Bull.}, 53(3):447--452, 2010.

\bibitem[CL]{CL}
Nhan-Phi Chung and Hanfeng Li.
\newblock Homoclinic groups, {IE} group, and expansive algebraic actions, 2011.
\newblock arXiv:1103.1567v1.

\bibitem[D1]{D1}
Christopher Deninger.
\newblock Deligne periods of mixed motives, {$K$}-theory and the entropy of
  certain {${\bf Z}^n$}-actions.
\newblock {\em J. Amer. Math. Soc.}, 10(2):259--281, 1997.

\bibitem[D2]{D2}
Christopher Deninger.
\newblock Fuglede-{K}adison determinants and entropy for actions of discrete
  amenable groups.
\newblock {\em J. Amer. Math. Soc.}, 19(3):737--758, 2006.

\bibitem[D3]{D3}
Christopher Deninger.
\newblock {$p$}-adic entropy and a {$p$}-adic {F}uglede-{K}adison determinant.
\newblock In {\em Algebra, arithmetic, and geometry: in honor of {Y}u. {I}.
  {M}anin. {V}ol. {I}}, volume 269 of {\em Progr. Math.}, pages 423--442.
  Birkh\"auser Boston Inc., Boston, MA, 2009.

\bibitem[D4]{D4}
Christopher Deninger.
\newblock Determinants on von {N}eumann algebras, {M}ahler measures and
  {L}japunov exponents.
\newblock {\em J. Reine Angew. Math.}, 651:165--185, 2011.

\bibitem[DS]{DS}
Christopher Deninger and Klaus Schmidt.
\newblock Expansive algebraic actions of discrete residually finite amenable
  groups and their entropy.
\newblock {\em Ergodic Theory Dynam. Systems}, 27(3):769--786, 2007.

\bibitem[Di]{Di}
Jacques Dixmier.
\newblock {\em von {N}eumann algebras}, volume~27 of {\em North-Holland
  Mathematical Library}.
\newblock North-Holland Publishing Co., Amsterdam, 1981.
\newblock With a preface by E. C. Lance, Translated from the second French
  edition by F. Jellett.

\bibitem[FL]{FL}
F.~Thomas Farrell and Peter~A. Linnell.
\newblock Whitehead groups and the {B}ass conjecture.
\newblock {\em Math. Ann.}, 326(4):723--757, 2003.

\bibitem[K]{K}
Yitzhak Katznelson.
\newblock {\em An introduction to harmonic analysis}.
\newblock Cambridge Mathematical Library. Cambridge University Press,
  Cambridge, third edition, 2004.

\bibitem[KLM]{KLM}
P.~H. Kropholler, P.~A. Linnell, and J.~A. Moody.
\newblock Applications of a new {$K$}-theoretic theorem to soluble group rings.
\newblock {\em Proc. Amer. Math. Soc.}, 104(3):675--684, 1988.

\bibitem[LSW]{LSW}
Douglas Lind, Klaus Schmidt, and Tom Ward.
\newblock Mahler measure and entropy for commuting automorphisms of compact
  groups.
\newblock {\em Invent. Math.}, 101(3):593--629, 1990.

\bibitem[La]{La}
Matilde~N. Lal{\'{\i}}n.
\newblock Mahler measures and computations with regulators.
\newblock {\em J. Number Theory}, 128(5):1231--1271, 2008.

\bibitem[Li]{Li}
Hanfeng Li.
\newblock Compact group automorphisms, addition formulas and
  {F}uglede--{K}adison determinants, 2010.
\newblock arXiv:1001.0419v1.

\bibitem[Lin]{Lin}
D.~A. Lind.
\newblock The entropies of topological {M}arkov shifts and a related class of
  algebraic integers.
\newblock {\em Ergodic Theory Dynam. Systems}, 4(2):283--300, 1984.

\bibitem[Lo]{Lo}
V.~Losert.
\newblock A characterization of groups with the one-sided {W}iener property.
\newblock {\em J. Reine Angew. Math.}, 331:47--57, 1982.

\bibitem[RV]{RV}
F.~Rodriguez Villegas.
\newblock Modular {M}ahler measures. {I}.
\newblock In {\em Topics in number theory ({U}niversity {P}ark, {PA}, 1997)},
  volume 467 of {\em Math. Appl.}, pages 17--48. Kluwer Acad. Publ., Dordrecht,
  1999.

\bibitem[Sch]{Sch}
Klaus Schmidt.
\newblock {\em Dynamical systems of algebraic origin}, volume 128 of {\em
  Progress in Mathematics}.
\newblock Birkh\"auser Verlag, Basel, 1995.

\bibitem[Sm]{Sm}
Chris Smyth.
\newblock The {M}ahler measure of algebraic numbers: a survey.
\newblock In {\em Number theory and polynomials}, volume 352 of {\em London
  Math. Soc. Lecture Note Ser.}, pages 322--349. Cambridge Univ. Press,
  Cambridge, 2008.

\bibitem[W]{W}
Norbert Wiener.
\newblock Tauberian theorems.
\newblock {\em Ann. of Math. (2)}, 33(1):1--100, 1932.

\end{thebibliography}

\end{document}